\documentclass[12pt]{article}%
\usepackage{amsmath}
\usepackage{amsfonts}
\usepackage{amssymb}
\usepackage{graphicx}%
\newtheorem{theorem}{Theorem}

\newtheorem{definition}[theorem]{Definition}
\newtheorem{example}[theorem]{Example}

\newtheorem{lemma}[theorem]{Lemma}

\newenvironment{proof}[1][Proof]{\noindent\textbf{#1.} }{\ \rule{0.5em}{0.5em}}
\begin{document}

\title{Best approximation from the positive cone of an inner product lattice}
\author{Marwen Abdouli\\{\small Laboratoire de Recherche LATAO}\\{\small D\'{e}partement de Mathematiques, Facult\'{e} des Sciences de Tunis}\\{\small Universit\'{e} de Tunis El Manar, 2092, El Manar, Tunisia}\\\texttt{{\small marouen.abdouli@ipest.ucar.tn}}
\and Karim Boulabiar\\{\small Laboratoire de Recherche LATAO}\\{\small D\'{e}partement de Mathematiques, Facult\'{e} des Sciences de Tunis}\\{\small Universit\'{e} de Tunis El Manar, 2092, El Manar, Tunisia}\\\texttt{{\small karim.boulabiar@fst.utm.tn}}}
\maketitle

\begin{abstract}
Let $E$ be a directed (i.e., positively generated) ordered vector space
endowed with an inner product. In this note, we prove that the following
statements are equivalent:

\begin{enumerate}
\item[(i)] $E$ is a vector lattice and its norm induced by its inner product
is a lattice norm.

\item[(ii)] The metric projection onto the positive cone $E^{+}$ of $E$ exists
and it is both isotone and subadditive.

\end{enumerate}

Moreover, in this case, the best approximation to any $x\in E$ from $E^{+}$
coincides with its positive part $x^{+}$.This result extends previous work to
the non-complete setting.

\end{abstract}

\bigskip

\section{Introduction}

In this paper, an \emph{inner product lattice} is defined as a real vector
lattice endowed with an inner product whose induced norm is a lattice norm.
Equivalently, it is a normed vector lattice in which the norm arises from an
inner product. When this norm is complete, the space is referred to as a
\emph{Hilbert lattice}.

In any normed vector lattice $E$, the positive cone $E_{+}$ is norm-closed
\cite{Z97}. Consequently, when $H$ is a Hilbert lattice, $H^{+}$ forms a
norm-complete convex set in $H$, and hence a Chebyshev set in $H$ in the sense
of Deutsch \cite{D01}. It follows that the metric projection $P_{H^{+}}$ onto
$H^{+}$ is automatically well-defined. Furthermore, it is even known that the
unique best approximation of any $x\in H$ from $H^{+}$ is precisely the
positive part $x^{+}$ of $x$. This striking result originates from the work of
Isac and N\'{e}meth \cite{IN90(a),IN89,IN90(b)}, who build substantially upon
a classical theorem by Moreau \cite{M62}. Applying again the Moreau's theorem,
and his aforementioned `best approximation' result with Isac, N\'{e}meth
established the following outstanding theorem.

Let $H$ be an ordered vector space whose positive cone $H^{+}$ is generating
and norm-closed (so that $P_{H^{+}}$ exists). Then the following conditions
are equivalent:

\begin{enumerate}
\item[(i)] $H$ is a Hilbert lattice.

\item[(ii)] $P_{H^{+}}$ is isotone and subadditive.
\end{enumerate}

A natural question arises concerning the non-complete setting. As far as we
are aware, this avenue of research remains unexamined in the existing
literature. Indeed, in general, the positive cone of an inner product lattice
need not be complete (see Example \ref{Exp} below), and therefore no \emph{a
priori} guarantee ensures the existence of the metric projection onto the
positive cone of an arbitrary inner product lattice. In the present work, we
provide a comprehensive solution to this problem. Specifically, we will
provide an elementary proof that the positive cone $E^{+}$ of a (not
necessarily complete) inner product lattice $E$ is a Chebyshev set in $E$, and
that the aforementioned `best approximation' result obtained by Isac and
N\'{e}meth remains valid in this broader context. In fact, we establish the
following substantially more general theorem:

Let $E$ be a directed ordered vector space---equivalently, a positively
generated one---equipped with an inner product. Then the following conditions
are equivalent:

\begin{enumerate}
\item[(i)] $E$ is an inner product lattice.

\item[(ii)] $E_{+}$ is a Chebyshev set in $E$, and $P_{E^{+}}$ is isotone and subadditive.
\end{enumerate}

\noindent In this setting, the equality $P_{E^{+}}\left(  x\right)  =x^{+}$
continues to hold for all $x\in E_{+}$. It is worth emphasizing that, beyond
extending the best approximation theorem of Isac and N\'{e}meth, our result
also generalizes N\'{e}meth's equivalence to the non-complete case.

To facilitate the reading of this paper, the reader is recommended to keep
close at hand books \cite{AB06,Z97} on vector lattices (also called Riesz
spaces) and book \cite{D01} on best approximation theory.

\section{Notation and terminology}

\begin{quote}
\textsl{Throughout the paper, we impose the standing assumption that all
vector spaces under consideration are real.}
\end{quote}

\noindent Let $E$ be a normed vector space and $A$ be a nonempty subset of
$E$. By a \emph{best approximation from }$A$ to an element $x\in E$, we mean
any element $x^{\ast}\in A$---if one exists---such that%
\[
\left\Vert x-x^{\ast}\right\Vert =\inf\left\{  \left\Vert x-a\right\Vert :a\in
A\right\}  .
\]
If such an element $x^{\ast}$ exists, we say that $x$ has a \emph{best
approximation in }$A$. We call $A$ a \emph{Chebyshev set} in $E$ if any $x\in
E$ has a unique best approximation in $A$. If $A$ is a Chebyshev set in $E$,
the unique best approximation of an element $x\in E$ is denoted by
$P_{A}\left(  x\right)  $. The map $P_{A}$, thus defined as the one that
assigns to each $x\in E$ its best approximation $P_{A}\left(  x\right)  $, is
called the \emph{metric projection} of $E$ onto $A$. Given that all these
conditions are fulfilled, it follows straightforwardly that if $x\in E$ then
$x\in A$ if and only if $P_{A}\left(  x\right)  =x$. In particular, $P_{A}$ is
an idempotent map. Further information on best approximations can be found in
the monograph \cite{D01} by Deutsch.

The following lines are devoted to the elementary theory of ordered vector
spaces and vector lattices. We adopt the terminology, notation, and basic
results from the classical references \cite{AB06} by Aliprantis and
Burkinshaw, and \cite{Z97} by Zaanen. As usual, the \emph{positive cone }of an
ordered vector space $E$ is denoted by $E^{+}$, i.e.,%
\[
E^{+}=\left\{  x\in E:0\leq x\right\}  .
\]
The ordered vector space $E$ is said to be \emph{directed} if, for any $x,y\in
E$, the inequalities $x\leq z$ and $y\leq z$ hold for some $z\in E$. It is not
hard to see that $E$ is directed if and only if it is \emph{positively
generated}, i.e., for every $x\in E$ there exist $y,z\in E^{+}$ such that
$x=y-z$. The ordered vector space $E$ is called a \emph{vector lattice }(or a
\emph{Riesz space}) if every pair of elements $x,y\in E$ has a least upper
bound (supremum) $x\vee y$ (or, equivalently, a greatest lower bound (infimum)
$x\wedge y$). The \emph{positive part }$x^{+}$ and the \emph{negative part
}$x^{-}$ of any element $x$ in the vector lattice $E$ are defined by%
\[
x^{+}=x\vee0\text{\quad and\quad}x^{-}=\left(  -x\right)  \vee0\text{,}%
\]
respectively. Notice that both $x^{+}$ and $x^{-}$ are elements in the
positive cone $E^{+}$ of $E$. Also, it is easily checked that%
\[
x^{+}\wedge x^{-}=0\text{ and }x=x^{+}-x^{-}\text{ for all }x\in E.
\]
In particular, any vector lattice is directed since it is positively
generated. If $x$ is an element in the vector lattice $E$, the supremum
$x\vee\left(  -x\right)  $ is called the \emph{absolute value} of $x$ and
denoted by $\left\vert x\right\vert $. One may check that%
\[
\left\vert x\right\vert =x^{+}+x^{-}\text{ for all }x\in E.
\]
Also, if $x,y\in E$ then $\left\vert x\right\vert \wedge\left\vert
y\right\vert =0$ if and only if $\left\vert x+y\right\vert =\left\vert
x-y\right\vert $ ($x$ and $y$ are said to be \emph{disjoint}).

Let $E$ be a vector lattice. A norm on $E$ is called a \emph{lattice norm} if%
\[
\left\vert x\right\vert \leq\left\vert y\right\vert \text{ in }E\quad
\text{implies\quad}\left\Vert x\right\Vert \leq\left\Vert y\right\Vert \text{
in }\mathbb{R}.
\]
A vector lattice along with a lattice norm is called a \emph{normed vector
lattice}. Observe that a norm on a vector lattice $E$ is a lattice norm if and
only if $\left\Vert x\right\Vert =\left\Vert \left\vert x\right\vert
\right\Vert $ for all $x\in E$ and $\left\Vert x\right\Vert \leq\left\Vert
y\right\Vert $ whenever $x\leq y$ in $E^{+}$.

Unless stated otherwise, all the above elementary properties will be assumed
and applied without further mention.

To conclude this brief section, we introduce the core concept of the paper.

\begin{definition}
\emph{A normed vector lattice }$E$ \emph{whose norm arises from an inner
product is called an} inner product lattice.
\end{definition}

In other words, an inner product lattice is a vector lattice endowed with an
inner product for which the induced norm is a lattice norm. If the inner
product lattice has a complete norm, it is called a \emph{Hilbert lattice}.
Equivalently, a Hilbert lattice is a Banach lattice whose norm derived from an
inner product.

\section{Best approximation from the positive cone}

As pointed out in the introduction, Isaac and N\'{e}meth proved that if $H$ is
a Hilbert lattice then the positive cone $H^{+}$ is a Chebyshev set in $H$ and%
\[
P_{H^{+}}\left(  x\right)  =x^{+}\text{ for all }x\in H
\]
(see Theorem 7 in \cite{IN90(b)}). This result was further refined by
N\'{e}meth \cite[Theorem 3]{N03}), who established that if $H$ is a Hilbert
space which is simultaneously a directed ordered vector space with a
norm-closed positive cone $H^{+}$ (and hence a Chebyshev set in $H$) then $H$
is Hilbert lattice if and only if the metric projection $P_{H^{+}}$ is isotone
(i.e., $P_{H^{+}}\left(  x\right)  \leq P_{H^{+}}\left(  y\right)  $ whenever
$x\leq y$) and subadditive (i.e., $P_{H^{+}}\left(  x+y\right)  \leq P_{H^{+}%
}\left(  x\right)  +P_{H^{+}}\left(  y\right)  $ for all $x,y\in H$).
Moreover, in this case, one has again%
\[
P_{H^{+}}\left(  x\right)  =x^{+}\text{ for all }x\in H.
\]
The main goal of this section is to extend these results to the more general
setting of an inner product lattice $E$, even when its positive cone $E^{+}$
is not complete --- a possibility illustrated by the following simple counterexample.

\begin{example}
\label{Exp}Let $E$ be the vector lattice of all real-valued continuous
functions on the real interval $[-1,1]$ \emph{(}usually denoted by $C\left(
\left[  -1,1\right]  \right)  $\emph{)}$\emph{.}$ The positive cone $E^{+}$ is
given by%
\[
E^{+}=\left\{  x\in E:x\left(  r\right)  \geq0\ \text{for all }r\in
\lbrack-1,1]\right\}
\]
Obviously, $E$ is an inner product lattice with respect to the inner product
defined by
\[
\left\langle x,y\right\rangle =\int_{-1}^{1}x\left(  r\right)  y\left(
r\right)  dr\text{ for all }x,y\in E.
\]
For every $n\in\{1,2,\ldots\}$, define $x_{n}\in E^{+}$ by%
\[
x_{n}\left(  x\right)  =\left(  \min\left\{  1,nr\right\}  \right)  ^{+}\text{
for all }r\in\left[  -1,1\right]  .
\]
A short moment's thought reveals that $(x_{n})$ is a Cauchy sequence in
$E^{+}$, but it does not converge in $E^{+}$.
\end{example}

We now require the following simple lemma to achieve our goal.

\begin{lemma}
\label{pos}Let $E$ be vector lattice equipped with an inner product. If
$E^{+}$ is a Chebyshev set in $E$ and $P_{E^{+}}\left(  x\right)  =x^{+}$ for
all $x\in E$ then%
\[
\left\langle x,y\right\rangle \in\left[  0,\infty\right)  \text{ for all
}x,y\in E^{+}.
\]

\end{lemma}

\begin{proof}
Assume that $E^{+}$ is a Chebyshev set in $E$ and that $P_{E^{+}}\left(
x\right)  =x^{+}$ for all $x\in E$. Pick $x,y\in E^{+}$ and observe that%
\[
P_{E^{+}}\left(  -x\right)  =x^{-}=0.
\]
So, by Theorem 4.1 in \cite{D01}, we obtain%
\[
\left\langle x,y\right\rangle =-\left\langle -x-P_{E^{+}}\left(  -x\right)
,y-P_{E^{+}}\left(  -x\right)  \right\rangle \geq0
\]
and the proof is complete.
\end{proof}

The proof of the central result of this section follows next.

\begin{theorem}
\label{M1}Let $E$ be a directed ordered vector space endowed with an inner
product. Then the following assertions are equivalent.

\begin{enumerate}
\item[\emph{(i)}] $E$ is an inner product lattice.

\item[\emph{(ii)}] $E^{+}$ is a Chebyshev set in $E$ and $P_{E^{+}}$ is
isotone and subadditive.
\end{enumerate}

\noindent In this case, $P_{E^{+}}\left(  x\right)  =x^{+}$ for all $x\in E$.
\end{theorem}

\begin{proof}
$\mathrm{(i)\Rightarrow(ii)}$ If $x\in E$ and $y\in E^{+}$ then from%
\[
-x\leq-x+y\leq\left\vert x-y\right\vert
\]
it follows that%
\[
x^{-}=\left(  -x\right)  ^{+}\leq\left\vert x-y\right\vert .
\]
But then%
\[
\left\Vert x-x^{+}\right\Vert =\left\Vert x^{-}\right\Vert \leq\left\Vert
\left\vert x-y\right\vert \right\Vert =\left\Vert x-y\right\Vert .
\]
It follows straightforwardly that $P_{E^{+}}$ exists and $P_{E^{+}}\left(
x\right)  =x^{+}$. Accordingly, if $x\leq y$ in $E$ then%
\[
P_{E^{+}}\left(  x\right)  =x^{+}\leq y^{+}=P_{E^{+}}\left(  y\right)  ,
\]
meaning that $P_{E^{+}}$ is isotone. Furthermore, if $x,y\in E$ then%
\[
P_{E^{+}}\left(  x+y\right)  =\left(  x+y\right)  ^{+}\leq x^{+}%
+y^{+}=P_{E^{+}}\left(  x\right)  +P_{E^{+}}\left(  y\right)
\]
and thus $P_{E^{+}}$ is subadditive.

$\mathrm{(ii)\Rightarrow(i)}$ First, choose $x\in E$. Since the ordered vector
space $E$ is directed (and so positively generated), the equality $x=y-z$
holds for some $x,y\in E^{+}$. Using the subadditivity of $P_{E^{+}}$, we get%
\begin{align*}
x  &  =y-z\\
&  =P_{E^{+}}\left(  y\right)  -P_{E^{+}}\left(  z\right) \\
&  =P_{E^{+}}\left(  x+z\right)  -P_{E^{+}}\left(  z\right) \\
&  \leq P_{E^{+}}\left(  x\right)  +P_{E^{+}}\left(  z\right)  -P_{E^{+}%
}\left(  z\right)  =P_{E^{+}}\left(  x\right)  .
\end{align*}
Now, let $y\in E^{+}$ such that $x\leq y$. This yields that $y-x\in E^{+}$ and
as $P_{E^{+}}$ is isotone, we derive that $P_{E^{+}}\left(  y\right)
-P_{E^{+}}\left(  x\right)  \in E^{+}$. It follows that%
\[
P_{E^{+}}\left(  x\right)  \leq P_{E^{+}}\left(  y\right)  =y.
\]
In summary, $P_{E^{+}}\left(  x\right)  $ is the supremum of $\left\{
0,x\right\}  $, meaning that $E$ is a vector lattice and $x^{+}=P_{E^{+}%
}\left(  x\right)  $. It remains to show that the norm induced by the inner
product is a lattice norm. To do this, observe that%
\[
\left\langle x^{+},x^{-}\right\rangle =0\text{ for all }x\in E.
\]
Indeed, if $x\in E$ then $4.7$ in \cite{D01} yields that%
\begin{align*}
\left\langle x^{+},x^{-}\right\rangle  &  =-\left\langle -x^{-},x^{+}%
\right\rangle =-\left\langle x-x^{+},x^{+}\right\rangle \\
&  =-\left\langle x-P_{E^{+}}\left(  x\right)  ,P_{E^{+}}\left(  x\right)
\right\rangle =0.
\end{align*}
It follows in particular, by Pythagore's Theorem, that%
\[
\left\Vert \left\vert x\right\vert \right\Vert ^{2}=\left\Vert x^{+}%
+x^{-}\right\Vert ^{2}=\left\Vert x^{+}-x^{-}\right\Vert ^{2}=\left\Vert
x\right\Vert ^{2}.
\]
Finally, if $0\leq x\leq y$ in $E$ then $0\leq y-x$ and so, using Lemma
\ref{pos} twice,%
\[
0\leq\left\langle y-x,x\right\rangle =\left\langle y,x\right\rangle
-\left\Vert x\right\Vert ^{2}\text{\quad and\quad}0\leq\left\langle
y-x,y\right\rangle =\left\Vert y\right\Vert ^{2}-\left\langle y,x\right\rangle
.
\]
This yields straightforwardly that $\left\Vert x\right\Vert \leq\left\Vert
y\right\Vert $, which completes the proof of the theorem.
\end{proof}

We conclude this paper with a concrete example that satisfies the conditions
of Theorem \ref{M1} yet is not covered by Isaac-N\'{e}meth and N\'{e}meth
theorems discussed above.

\begin{example}
The vector lattice of all real-valued continuous functions on the real
interval $\left[  0,1\right]  $ is denoted by $C\left(  \left[  0,1\right]
\right)  $, as sual. Let $E$ denote the vector sublattice of $C\left(  \left[
0,1\right]  \right)  $ of all piecewise polynomial functions. From the density
of the set $\{\cos n:n\in\left\{  0,1,2,....\right\}  \}$ in $[0,1]$ it
follows quickly that the formula
\[
\left\langle x,y\right\rangle =\sum_{n=0}^{+\infty}\frac{x(\cos n)y(\cos
n)}{2^{n}}%
\]
turns $E$ into an inner product lattice. Hence, by \emph{Theorem \ref{M1}},
the positive cone $E^{+}$ is a Chebyshev set and%
\[
P_{E^{+}}(x)=x^{+}\text{ for all }x\in E.
\]
Notice here that $E$ is far from being norm-complete.
\end{example}

\end{document}